\documentclass[11pt]{amsart}

\usepackage{amsmath, amssymb, amsthm, amsfonts, enumerate, color, comment}
\usepackage{mathrsfs}
\usepackage{parskip}
\usepackage{xcolor}

\usepackage{mathtools}
\mathtoolsset{showonlyrefs}

\usepackage{palatino}
\usepackage{graphicx}

\usepackage{parskip}

\numberwithin{equation}{section}
\theoremstyle{plain}
\newtheorem{Proposition}[equation]{Proposition}
\newtheorem{Corollary}[equation]{Corollary}
\newtheorem*{Corollary*}{Corollary}
\newtheorem{Theorem}[equation]{Theorem}
\newtheorem*{Theorem*}{Theorem}
\newtheorem{Lemma}[equation]{Lemma}
\theoremstyle{definition}

\newtheorem{Example}[equation]{Example}
\newtheorem{Remark}[equation]{Remark}
\newtheorem{Question}[equation]{Question}

\usepackage{enumitem}
\setlist[enumerate]{leftmargin=*}
\setlist[itemize]{leftmargin=*}

\setlist[enumerate,1]{label=(\alph*),font=\upshape}

\setlist[enumerate,2]{label=(\roman*),font=\upshape}

\def\C{\mathbb{C}}

\def\D{\mathbb{D}}

\def\Z{\mathbb{Z}}

\renewcommand{\leq}{\leqslant}
\renewcommand{\geq}{\geqslant}
\renewcommand{\subset}{\subseteq}
\renewcommand{\phi}{\varphi}
\renewcommand{\vec}[1]{{\bf #1}}
 \renewcommand{\Re}[1]{\operatorname{Re} #1 }

\usepackage{xcolor}

	\author[A. Belli]{Anil Belli}
	\address{Department of Mathematics, University of Thessaloniki, 54124 Thessaloniki,
Greece}
	\email{anilbelli@math.auth.gr}
	
	\author[U. Gul]{Ugur Gul}
	\address{Hacettepe University, Department of Mathematics, 06800, Beytepe, Ankara, Turkey}
	\email{gulugur@gmail.com}

\author[W. Ross]{William T. Ross}
	\address{Department of Mathematics and Computer Science, University of Richmond, Richmond, VA 23173, USA}
	\email{wross@richmond.edu}
	
	\author[A. Siskakis]{Aristomenis G. Siskakis}
	\address{Department of Mathematics, University of Thessaloniki, 54124 Thessaloniki,
Greece}
	\email{siskakis@math.auth.gr}
	

\title{The Ces\`{a}ro operator on $L^2(0, 1)$}


\begin{document}

\begin{abstract}
This paper explores a version of the classical Ces\`{a}ro integral operator for the Lebesgue space $L^2(0, 1)$ where we discuss its norm, adjoint, spectral properties, and invariant subspaces. An important tool will be semigroups of weighted composition operators on $L^2(0, 1)$.
\end{abstract}

\maketitle

\section{Introduction}

There is the well-known and well studied {\em Ces\`{a}ro operator }$\mathscr{C}$ on the classical Hardy space $H^2$ \cite{Duren} of the open unit disk $\D:= \{z: |z| < 1\}$ defined by 
\begin{equation}\label{Adf88UasA}
(\mathscr{C}f)(z) = \frac{1}{z} \int_{0}^{z} \frac{f(\xi)}{1 - \xi} d \xi, \quad z \in \D,
\end{equation}
where a 1965 paper of Brown, Halmos, and Shields \cite{MR187085} showed that $\mathscr{C}$ is a bounded linear operator on $H^2$ satisfying
$$\|\mathscr{C}\| = 2, \; \; \sigma(\mathscr{C}) = \overline{D(1, 1)}, \; \;  \sigma_{p}(\mathscr{C}) = \varnothing, \; \sigma_{p}(\mathscr{C}^{*}) = D(1, 1),$$ 
$$(\mathscr{C}^{*} f)(z) = \frac{1}{1 -z} \int_{z}^{1} f(\xi) d \xi,$$
and $\mathscr{C}$ is a hyponormal operator on $H^2$ (i.e., $\mathscr{C}^{*} \mathscr{C} - \mathscr{C} \mathscr{C}^{*} \geq 0$). In the above, $\|\mathscr{C}\|$ denotes the operator norm of $\mathscr{C}$, $D(a, r) := \{z: |z - a| < r\} $, $\sigma(\mathscr{C})$ denotes the spectrum of $\mathscr{C}$, $\sigma_{p}(\mathscr{C})$ denotes its point spectrum, and $\mathscr{C}^{*}$ denotes the adjoint of $\mathscr{C}$. Throughout this paper, $\overline{A}$ will denote the closure of a set $A \subset \C$. The matrix representation of $\mathscr{C}$ with respect to the orthonormal basis $\{z^n: n \geq 0\}$ for the Hardy space $H^2$ is the famous Ces\`{a}ro matrix
$$\begin{bmatrix}
1 \;  & 0 \; & 0 \; & 0 \; & 0  \; & \! \cdots\\[2.5pt]
\tfrac{1}{2} \; & \tfrac{1}{2} \; & 0 \; & 0 \; & 0 \; &  \cdots\\[2.5pt]
\frac{1}{3} \; & \frac{1}{3} \; & \frac{1}{3} \; & 0 \; & 0  \; &  \cdots\\[2.5pt]
\frac{1}{4} \; & \frac{1}{4} \; & \frac{1}{4} \; & \frac{1}{4} \; & 0 \; &  \cdots\\[2.5pt]
\frac{1}{5} \; & \frac{1}{5} \; & \frac{1}{5} \; & \frac{1}{5} \; & \frac{1}{5} \; &  \!\cdots\\[1pt]
 \vdots & \vdots & \vdots & \vdots & \vdots  & \ddots
\end{bmatrix}.$$
In fact, the original Ces\`{a}ro operator explored in  \cite{MR187085} was given as a matrix operator on the sequence space $\ell^2$ but, due to the natural isometric isomorphism between $\ell^2$ and the Hardy space $H^2$,  it can be realized as an integral operator on the analytic function space $H^2$.

This work was continued by Kriete and Trutt \cite{MR281025} who observed in 1971 that $\mathscr{C}$ is cyclic (i.e., the orbit $\{\mathscr{C}^n \chi_{\D}: n \geq 0\}$ of $\chi_{\D}$,  the constant function equal to one on $\D$, has dense linear span in $H^2$) and, somewhat surprisingly,  a subnormal operator (i.e., has a normal extension). Still further, the invariant subspaces of the Ces\`{a}ro operator were studied in  \cite{GPR, MR4757014, MR350489}. Due to their complexity and their abundance, the invariant subspaces for $\mathscr{C}$ can be quite complicated and our current understanding of them is incomplete.  Since these initial papers, subsequent authors  have examined the Ces\`{a}ro operator, and operators  closely related to it,  on other Hilbert and Banach spaces of analytic functions on $\D$ \cite{MR5037546}. 

This paper explores  the analogous properties of the real variable {\em Ces\`{a}ro operator} $C$ on the classical Lebesgue space  $L^2(0, 1)$ defined as 
\begin{equation}\label{66TzzzCesaroewal}
(C f)(x) := \frac{1}{x} \int_{0}^{x} \frac{f(t)}{1 - t} dt, \quad 0 < x < 1.
\end{equation}
Of course one sees how this formula is inspired by the corresponding integral  formula for the classical Ces\`{a}ro operator $\mathscr{C}$ on $H^2$. Observe how $\mathscr{C}$ acts on a space of analytic functions on $\D$ and so  the integration is along any (rectifiable) path in $\D$ from $\xi = 0$ to $\xi = z$ while $C$ acts on the Lebesgue space of functions on $(0, 1)$ and so  the path of integration is along the interval $(0, 1)$ from $t = 0$ to $t = x$. 
To add a bit of historical context, the  Ces\`{a}ro operator from \eqref{66TzzzCesaroewal}  is amongst a class of well-known integral operators on $L^2(0, 1)$ such as 
the {\em Volterra operator}
\begin{equation}\label{Volt}
(V f)(x)= \int_{0}^{x} f(t)dt
\end{equation}
and the {\em Hardy averaging operator }
\begin{equation}\label{Haaa}
(H f)(x) = \frac{1}{x} \int_{0}^{x} f(t) dt,
\end{equation}
where  their norms, spectra, and invariant subspaces have been explored \cite{MR31110, MR92124, GPRH, MR4545809, MR5037546, MR3793153}. As far as we know, the Ces\`{a}ro operator  $C$ on $L^2(0, 1)$ has not been previously studied. There is also an  issue of nomenclature. Some authors  \cite{MR187085} call the Hardy operator $H$ the ``continuous Ces\`{a}ro operator'' while others call $H$ the ``real variable Ces\`{a}ro operator''. As to not invite confusion, we will do our best to be specific and always call $C$ the Ces\`{a}ro operator on $L^{2}(0, 1)$.

In this paper we show that the Ces\`{a}ro operator $C$ from \eqref{66TzzzCesaroewal} defines a bounded linear operator on $L^2(0, 1)$ satisfying 
$$\|C\| = 2, \; \sigma(C) = \partial D(1, 1), \;  \sigma_{p}(C) = \sigma_{p}(C^{*}) = \varnothing,$$
and 
$$(C^{*} f)(x) = \frac{1}{1 - x} \int_{x}^{1} \frac{f(t)}{t}dt.$$
In the above, $\partial D(1, 1) = \{z: |z - 1| = 1\}$. Along the way, we compute a formula for the resolvent operator $(\lambda I - C)^{-1}, \lambda \not \in \partial D(1, 1)$. 
 We also show that $C$ is a cyclic and   {\em normal} operator on $L^2(0, 1)$ (i.e., $C^{*}C = C C^{*}$). Note the differences between the properties of the Ces\`{a}ro operator $C$ on $L^2(0, 1)$ and those for the classical Ces\`{a}ro operator $\mathscr{C}$ on $H^2$. 
The  spectral theorem for cyclic normal Hilbert space operators \cite[p.~275]{ConwayFA} implies that $C$ is unitarily equivalent to the multiplication operator $M_{z} f = z f$ on $L^2(\mu)$ for some positive finite Borel measure $\mu$ on $\sigma(C) = \partial D(1, 1)$. We identify this measure $\mu$ as well as the intertwining isometric isomorphism yielding this equivalence. 
  Via some semigroup techniques, described below,  along with the Beurling--Lax theorem \cite[Lec.~V]{MR0171178} \cite[p.~204]{MR3890074}, which explores the translation invariant subspaces of $L^2(-\infty, \infty)$, we also completely describe the invariant subspaces for $C$. 

An important tool in our analysis will be an associated strongly continuous semigroup $\{S_{t}\}_{t \geq 0}$ of weighted composition operators on $L^2(0, 1)$. Analogous to results  for the classical Ces\`{a}ro operator \cite{MR4757014} and the Hardy averaging operator \cite{GPRH}, we will show that a closed subspace $\mathcal{M}$ of $L^2(0, 1)$ is invariant for $C$, i.e., $C \mathcal{M} \subset \mathcal{M}$, if and only if $\mathcal{M}$ is invariant for each element of the semigroup $\{S_t\}_{t \geq 0}$. Each $S_t$ is an interesting operator in its own right and we will represent it as a bilateral shift operator on a vector-valued $\ell^2(\Z)$ space as was explored in \cite{MR2858500, MR3057690} in a related setting involving weighted composition operators on $L^2(0, 1)$. Linearizing the semigroup $\{S_{t}\}_{t \geq 0}$ via an isometric isomorphism from $L^2(0, 1)$ onto $L^2(-\infty, \infty)$ will convert it to a semigroup of translations $\{T_t\}_{t \geq 0}, (T_{t} f)(x)  = f(x + t)$, where the Beurling--Lax theorem allows us to characterize the invariant subspaces of the Ces\`{a}ro operator on $L^2(0, 1)$. 

A subsequent paper \cite{CLp} will expand our discussion of the Ces\`{a}ro operator on $L^2(0, 1)$ to  $L^p(0, 1)$ for $ 1 \leq  p <  \infty$, where the results become more complicated. 

The authors would like to thank Benoit F. Sehba for pointing out a small error in an earlier version of this paper. 

\section{Norm, Spectrum, and Normality}

Below we will use 
$$\|f\| := \Big( \int_{0}^{1} |f(x)|^2 dx\Big)^{\frac{1}{2}}$$ to denote the norm on $L^2(0, 1)$ and $\langle f, g\rangle$  the corresponding inner product. Our first order of business is to establish that the Ces\`{a}ro operator $C$ on $L^2(0, 1)$ is bounded. This can be accomplished in several ways. We will use Schur's test \cite[p.~77]{MR4545809}. 

\begin{Proposition}[Schur's test]
Let $(X, \mu)$ be a measure space and $\mu$ a positive measure on $X$. If $k(x, y)$ is a nonnegative measurable function on $X \times X$, $p(x)$, $q(x)$ are a nonnegative measurable functions on $X$, and $\alpha, \beta$ are positive constants such that 
$$\int_{X} k(x, y) p(y) d\mu(y) \leq \alpha q(x) \; \text{for $\mu$-almost every $x \in X$}$$
and
$$\int_{X} k(x, y) q(x) d\mu(x) \leq \beta p(y) \; \text{for $\mu$-almost every $y \in X$},$$ then 
$$(T f)(x) = \int_{X} k(x, y) f(y) d \mu(y)$$ defines a bounded linear operator on $L^2(\mu)$ with $\|T\| \leq \sqrt{\alpha \beta}$.
\end{Proposition}

Schur's test enables us to prove the boundedness of $C$ on $L^2(0, 1)$.

\begin{Theorem}
The Ces\`{a}ro operator $C$ is bounded on $L^2(0, 1)$. 
\end{Theorem}

\begin{proof}
Apply Schur's test to $L^2(0, 1)$ with the parameters
\begin{equation}\label{kkkKKX}
k(x, t) := \begin{cases}
0 & \mbox{if $x \leq t \leq 1$},\\[5pt]
{\displaystyle \frac{1}{x} \frac{1}{1 - t}} & \mbox{if $0 < t \leq x$},
\end{cases}
\end{equation}
$$p(x) = q(x) = \frac{1}{\sqrt{x (1 - x)}} , \quad \alpha = \beta = 2.$$
Then an integral computation verifies that 
$$
\int_{0}^{1} k(x, t) p(t) dt  = 2 q(x)$$
and another integral calculation gives us 
$$\int_{0}^{1} k(x, t) q(x) dx  =  2 p(t).$$
By Schur's test, $C$ is bounded on $L^2(0, 1)$ with $\|C\| \leq 2$. 
\end{proof}

In Theorem \ref{boundedCCC} below, we will sharpen  the inequality above and show that  $\|C\| = 2$. But first we need a formula for the adjoint of $C$. 

\begin{Theorem}\label{adskksdf}
For the Ces\`{a}ro operator $C$ on $L^2(0, 1)$, its adjoint $C^{*}$ satisfies 
$$(C^{*} f)(x) = \frac{1}{1 - x} \int_{x}^{1} \frac{f(t)}{t}dt, \quad 0 < x < 1.$$
\end{Theorem}

\begin{proof}  The formula from \eqref{kkkKKX}  says that 
$$(C f)(x) = \int_{0}^{1} k(x, t) f(t) dt$$
and thus, from basic facts about integral operators (which can be verified quite easily), 
$$(C^{*} f)(x) = \int_{0}^{1} k^{*}(x, t) f(t) dt,$$
where 
\begin{equation}\label{k*}
k^{*}(x, t) := 
\begin{cases}
0 & \mbox{if $0 \leq t \leq x$},\\[5pt]
{\displaystyle \frac{1}{t} \cdot \frac{1}{1 - x}} & \mbox{if $0 < x \leq t$}.
\end{cases}
\end{equation}
This yields the desired adjoint formula.
\end{proof}

We now proceed to compute the norm and spectrum of the Ces\`{a}ro operator $C$ on $L^2(0, 1)$ along with proving its normality. We begin by defining the change of variables 
\begin{equation}\label{xxxuuu}
u = \log \frac{1 - x}{x}, \quad x(u) = \frac{1}{1 + e^u}
\end{equation} and observe that the map  $u \mapsto x(u)$ defines a diffeomorphism of $(-\infty, \infty)$ onto $(0, 1)$. 

\begin{Proposition}\label{sdfsdf88}
If $f \in L^2(0, 1)$ and $\Phi f$ is the function on $(-\infty, \infty)$ defined by
$$(\Phi f)(u) := \frac{e^{\frac{u}{2}}}{1 + e^{u}} f \Big(\frac{1}{1 + e^{u}}\Big), \quad -\infty < u < \infty,$$
then the  map $f \mapsto \Phi f$ is an isometric isomorphism from $L^2(0, 1)$ onto $L^2(-\infty, \infty)$. 
\end{Proposition}

\begin{proof}
Since $x(u) = (1 + e^{u})^{-1}$  and $u \mapsto x(u)$ maps $(-\infty, \infty)$ smoothly onto $(0, 1)$, we can write $\Phi f$ as 
$$(\Phi f)(u) = f(x(u)) \sqrt{|x'(u)|}.$$ Thus,  an integral change of variables yields 
\begin{align*}
\|\Phi f\|^{2}_{L^2(-\infty, \infty)} & = \int_{-\infty}^{\infty} |f(x(u))|^2 |x'(u)| du\\
& = \int_{0}^{1} |f(x)|^2 dx,
\end{align*}
proving that $f \mapsto \Phi f$ is a linear  isometry. The surjectivity of $\Phi$ comes from the fact that $x(u)$ is a diffeomorphism. In fact, 
\begin{equation}\label{gginnd}
(\Phi^{-1} g)(x) = \frac{1}{x} \sqrt{\frac{x}{1 - x}} g\Big(\log \frac{1 - x}{x}\Big), \quad 0 < x < 1, \quad g \in L^{2}(-\infty, \infty),
\end{equation}
which proves the result.
\end{proof}

Conjugating $C$ by the isometric isomorphism $\Phi$ produces convolution operator on $L^2(-\infty, \infty)$ in the following way.
The computation below will use the change of variables 
$$s = \log \frac{1 - t}{t}, \quad ds = -\frac{1}{t(1 - t)} dt.$$
Proposition \ref{sdfsdf88} says that 
$\Phi \circ C \circ \Phi^{-1}: L^2(-\infty, \infty) \to L^2(-\infty, \infty)$ and so  for $f \in L^2(-\infty, \infty)$ and almost every $u \in (-\infty, \infty)$ we have 
\begin{align*}
(\Phi \circ C \circ \Phi^{-1} f)(u) & = \Phi \circ C \left(\frac{1}{x} \sqrt{\frac{x}{1 - x}} f\Big(\log \frac{1 - x}{x}\Big)\right)(u)\\
& = \Phi \left(\frac{1}{x} \int_{0}^{x} \frac{1}{t} \sqrt{\frac{t}{1 - t}} f\Big(\log \frac{1 - t}{t}\Big) \frac{1}{1 - t}dt\right)(u)\\
& = \Phi \left(\frac{1}{x} \int_{\log \frac{1 - x}{x}}^{\infty} e^{-\frac{s}{2}} f(s) ds\right)(u)\\
& = \frac{e^{\frac{u}{2}}}{1 + e^{u}} (1 + e^{u}) \int_{u}^{\infty} e^{-\frac{s}{2}} f(s) ds\\
& = \int_{u}^{\infty} e^{\frac{u - s}{2}} f(s) ds.
\end{align*}
In summary, $\Phi \circ C \circ \Phi^{-1}$ can be written as a convolution operator 
$$(\Phi \circ C \circ \Phi^{-1} f)(x) = (G \ast f)(x) := \int_{-\infty}^{\infty} G(x - t) f(t) dt,$$
where 
$$G(t) := \chi_{(-\infty, 0)}(t) e^{\frac{t}{2}}.$$
We now convert the Ces\`{a}ro operator  $C$ on $L^2(0, 1)$ into a multiplication operator on $L^2(-\infty, \infty)$ via the Fourier transform 
\begin{equation}\label{FFFFTTTT}
 (\mathscr{F} g)(s) := \frac{1}{\sqrt{2 \pi}} \int_{-\infty}^{\infty} g(x) e^{-i x s} dx, \quad g \in L^2(-\infty, \infty),
 \end{equation}
  on $L^2(-\infty, \infty)$ as follows. Standard theory of Fourier transforms (see \cite[p.~100]{MR2449250}  or \cite[Ch.~11]{MR4545809}) says that $\mathscr{F}$ is an isometric isomorphism of $L^2(-\infty, \infty)$ onto itself. Moreover, for any $f \in L^2(-\infty, \infty)$ we have 
 \begin{align*}
 \mathscr{F} ((\mathscr{F}^{-1} f) \ast G)(x) & = \sqrt{2 \pi}\, (f  \cdot \mathscr{F} G)(x)\\
 & = \sqrt{2 \pi}\, f(x) \cdot \frac{1}{\sqrt{2 \pi}} \int_{-\infty}^{0} e^{\frac{t}{2} - i x t} dt\\
 & = f(x) \cdot \frac{1}{\frac{1}{2} - i x}.
 \end{align*}
Putting this all together says that 
\begin{equation}\label{sd8f8sdfhhhMMMMMmFF}
(\mathscr{F} \circ \Phi) \circ C \circ (\mathscr{F} \circ \Phi)^{-1} = M_{(\frac{1}{2} - i x)^{-1}},
\end{equation} which yields the following result. 

\begin{Theorem}
The Ces\`{a}ro operator on $L^2(0, 1)$ is unitarily equivalent to the multiplication operator 
$$(M_{(\frac{1}{2} - i x)^{-1}} g)(x) = \frac{1}{\frac{1}{2} - i x} \cdot g(x)$$ on $L^2(-\infty, \infty)$. 
\end{Theorem}

Now use the fact that the norm of a multiplication operator is the (essential) supremum norm of its symbol and its spectrum is the (essential) range of its symbol \cite[Ch.~8]{MR4545809}  to conclude that 
$$\|M_{(\frac{1}{2} - i x)^{-1}}\| = \sup_{x \in (-\infty, \infty)} \Big|\frac{1}{\frac{1}{2} - i x}\Big| = 2$$
and 
$$ \sigma(M_{(\frac{1}{2} - i x)^{-1}})  = \operatorname{clos}\Big\{\frac{1}{\frac{1}{2} - i x}: x \in (-\infty, \infty)\Big\} = \partial D(1, 1).$$
The unitary equivalence of the operators $C$ and $M_{(\frac{1}{2} - i x)^{-1}}$ via $\mathscr{F} \circ \Phi$ implies the following. 

\begin{Theorem}\label{boundedCCC}
For the Ces\`{a}ro operator $C$ on $L^2(0, 1)$ we have 
$\|C\| = 2$ and  $\sigma(C) = \partial D(1, 1).$
\end{Theorem}

The known normality of  multiplication operators on $L^2(-\infty, \infty)$ \cite[Ch.~8]{MR4545809}, along with the above discussion yields the normality of $C$.

\begin{Theorem}
The Ces\`{a}ro operator $C$ on $L^2(0, 1)$ is normal.
\end{Theorem}

An alternative proof of the normality of $C$, inspired by a similar calculation  from \cite{MR187085} involving the Hardy averaging operator (recall \eqref{Haaa}), comes from the fact that the integral kernel for the operator $C C^{*}$ (recall the kernels $k$ and $k^{*}$ from \eqref{kkkKKX} and \eqref{k*}) is 
\begin{align*}
\int_{0}^{1} k(x, u) k^{*}(u, t) du & =  \int_{0}^{\min(x, t)} \frac{1}{x}  \frac{1}{1 - u}  \frac{1}{1 - u} \frac{1}{t} du\\
& = \frac{1}{t x} \Big(\frac{1}{1 - \min(x, t)} - 1\Big),
\end{align*}
while the integral kernel for the operator $C^{*} C$ is 
\begin{align*}
\int_{0}^{1} k^{*}(x, u) k(u, t)du &= \int_{\max(x, t)}^{1} \frac{1}{1 - x}  \frac{1}{u} \frac{1}{u} \frac{1}{1 - t} du\\
& = \frac{1}{(1 - t)(1 - x)} \Big(\frac{1}{\max(x, t)} - 1\Big).
\end{align*}
By considering the cases where $x \geq t$ and $t \geq x$ separately, one can show these two expressions are equal and thus $C C^{*} = C^{*} C$. The normality of $C$ can also be verified by computing $C^{*}C$ and $C C^{*}$ on the monomials $\{x^n: n \geq 0\}$ and using identities involving the incomplete beta function.

\begin{Remark}\label{RRR}
Another curiosity arising from the above analysis is the formula 
$$(I - C)^{-1} = I - C^{*}$$
which can be seen by first observing that 
\begin{align*}
(I - M_{(\frac{1}{2} - i x)^{-1}})(I - M_{(\frac{1}{2} - i x)^{-1}})^{*} & = M_{\frac{-i x - 1/2}{-i x + 1/2}}  M_{\frac{-i x - 1/2}{-i x + 1/2}}^{*}\\
& =  M_{\frac{-i x - 1/2}{-i x + 1/2}}  M_{\frac{i x - 1/2}{i x + 1/2}}\\
& = I.
\end{align*}
Now conjugate the above identity  with the intertwining map  $\mathscr{F} \circ \Phi$ discussed earlier  to obtain
$$(I - C) (I - C^{*}) = I,$$
and, in a similar way, $(I - C^{*})(I - C) = I$. 

Since 
$$M_{(\frac{1}{2} - i x)^{-1}} M_{(\frac{1}{2} - i x)^{-1}}^{*} = M_{(\frac{1}{2} - i x)^{-1}} + M_{(\frac{1}{2} - i x)^{-1}}^{*},$$
it follows, again  via  conjugation by $\mathscr{F} \circ \Phi$, that 
$$C C^{*} = C + C^{*}$$
and likewise $C^{*}C = C + C^{*}$, again confirming the normality of $C$. 
\end{Remark}

As it turns out, the point spectrum $\sigma_{p}(C)$, i.e., the set of eigenvalues of $C$, is not very interesting. Neither is $\sigma_{p}(C^*)$. 

\begin{Proposition}\label{eigen2}
For the Ces\`{a}ro operator $C$ on $L^2(0, 1)$ we have 
$$\sigma_{p}(C) = \sigma_{p}(C^{*}) = \varnothing.$$
\end{Proposition}

\begin{proof}
The multiplication operators 
$$M_{(\frac{1}{2} - i x)^{-1}} \; \mbox{and} \; M^{*}_{(\frac{1}{2} - i x)^{-1}} = M_{(\frac{1}{2} + i x)^{-1}}$$ on $L^2(-\infty, \infty)$ have no eigenvalues. Conjugating by $\mathscr{F} \circ \Phi$ shows that $C$ and $C^{*}$ have no eigenvalues either. 
\end{proof}

\section{Spectral Representation}

Recall that a bounded  operator $T$ on a (separable) Hilbert space $\mathcal{H}$  is {\em cyclic} if there is an $\vec{x} \in \mathcal{H}$ such that 
$$\overline{\operatorname{span}}\{T^{n} \vec{x}: n \geq 0\} = \mathcal{H}.$$
In other words, the closed linear span of the orbit of $\vec{x}$ under $T$ is all of $\mathcal{H}$. 
The vector $\vec{x}$ is called a {\em cyclic vector} for $T$. 

\begin{Proposition}\label{cvbNNN}
The Cesaro operator $C$ on $L^2(0, 1)$ is cyclic with cyclic vector $\chi_{[0, 1]}$.
\end{Proposition}

\begin{proof}
It is known \cite[p.~203]{MR350489} that the classical Ces\`{a}ro operator $\mathscr{C}$ is cyclic on the Hardy space $H^2$ with cyclic vector $\chi_{\D}$. Therefore, given any $\varepsilon > 0$ and $q \in \C[z]$, there is a $p \in \C[z]$ such that 
\begin{equation}\label{FRRee}
\int_{0}^{2 \pi} |(p(\mathscr{C}) \chi_{\D})(e^{i \theta}) - q(e^{i \theta})|^2 \frac{d \theta}{2 \pi} < \varepsilon.
\end{equation}
The above integral represents the square of the $H^2$ norm of $p(\mathscr{C})\chi_{\D} - q$ \cite[p.~21]{Duren}. 

We pause for a moment to remark that a calculation involving the  polylogarithm function shows  that for any integer $n \geq 0$, 
$$(\mathscr{C}^n \chi_{\D})(z) = \sum_{k = 1}^{\infty} \frac{z^k}{k^n}, \quad z \in \D,$$
and 
$$(C^n \chi_{[0, 1]})(x) = \sum_{k = 1}^{\infty} \frac{x^k}{k^n}, \quad 0 < x < 1.$$
Thus,
\begin{equation}\label{poly}
(\mathscr{C}^n \chi_{\D})(x) = (C^n \chi_{[0, 1]})(x), \quad 0 < x < 1,
\end{equation}
for all integers $n \geq 0$ and so 
$$p(\mathscr{C} \chi_{\D})(x) = p(C \chi_{[0, 1]})(x)$$ for all $p \in \C[z]$.

The Fej\'{e}r--Riesz inequality \cite[p.~46]{Duren} says that 
$$\int_{-1}^{1} |g(x)|^2 dx \leq \tfrac{1}{2} \int_{0}^{2 \pi} |g(e^{i \theta})|^2 d \theta, \quad g \in H^2,$$ which, applied to \eqref{FRRee} and \eqref{poly}, yields 
\begin{align*}
\int_{0}^{1} |(p(C) \chi_{[0, 1]})(x) - q(x)|^2 dx & \leq \pi \int_{0}^{2 \pi}  |(p(\mathscr{C}) \chi_{\D})(e^{i \theta}) - q(e^{i \theta})|^2 \frac{d \theta}{2 \pi}\\
& < \pi \varepsilon.
\end{align*}
This implies that 
$q \in \overline{\operatorname{span}}\{C^{n} \chi_{[0, 1]}: n \geq 0\}$ for any $q \in \C[x]$.
The known density of the polynomials in $L^2(0, 1)$ completes the proof. 
\end{proof}

\begin{Question}
What are other cyclic vectors for $C$ on $L^2(0, 1)$? Can they be characterized?
\end{Question}

Since $C$ is both normal and cyclic, the spectral theorem says there is a finite positive Borel measure $\mu$ on $\sigma(C) = \partial D(1, 1)$ such that $C$ is unitarily equivalent to the multiplication operator $M_z f = z f$ on $L^2(\mu)$  \cite[p.~275]{ConwayFA}. What is $\mu$? What is the  intertwining isometric isomorphism between $L^2(0, 1)$ and $L^2(\mu)$ that yields this unitary equivalence? We now answer these questions with the following discussion. 

Apply the isometric isomorphism $\Phi: L^2(0, 1) \to L^2(-\infty, \infty)$ from Proposition \ref{sdfsdf88} to the cyclic vector $\chi_{[0, 1]}$ for $C$ (see the proof of Proposition \ref{cvbNNN}) to get
$$(\Phi \chi_{[0, 1]})(u) = \frac{1}{2 \cosh(\frac{u}{2})}, \quad -\infty < u < \infty.$$
Taking Fourier transforms of the previous equation yields 
$$(\mathscr{F} \circ \Phi \chi_{[0, 1]})(x) = \sqrt{\frac{\pi}{2}} \operatorname{sech} (\pi x), \quad -\infty < x < \infty.$$
Now define a measure $\nu$ on $(-\infty, \infty)$ by 
$$d \nu(x) = |(\mathscr{F} \circ \Phi \chi_{[0, 1]})(x)|^2 dx = \tfrac{\pi}{2} \operatorname{sech}^2(\pi x) dx$$
(see Figure \ref{Fig_Sec})
\begin{figure}
\begin{center}
 \includegraphics[width=.4\textwidth]{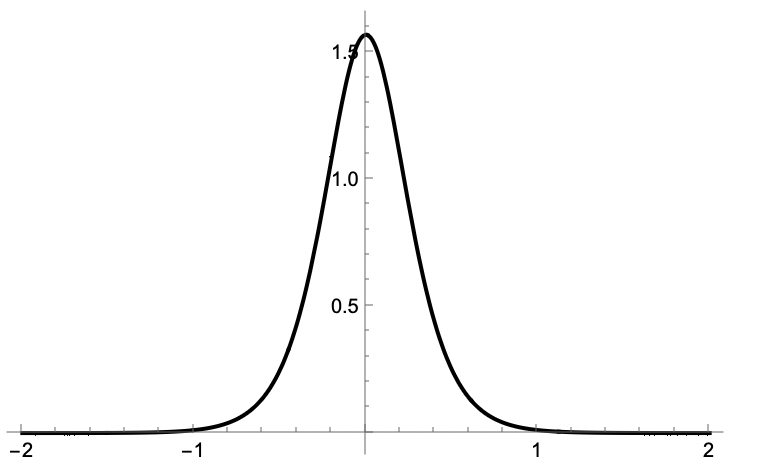}
 \caption{The graph of $\frac{\pi}{2} \operatorname{sech}^2(\pi x)$. }
 \label{Fig_Sec}
 \end{center}
\end{figure}
and the push forward measure $\mu$ on $\partial D(1, 1)$ by  
\begin{equation}\label{uuWwWW}
\int_{\partial D(1, 1)} g(z) d \mu(z) = \int_{-\infty}^{\infty} (g \circ m)(x) d \nu(x), \quad g \in C(\partial D(1, 1)).
\end{equation}
In the above, 
$$m(x) := \frac{1}{\frac{1}{2}- i x}, \quad -\infty < x < \infty,$$
and  maps $(-\infty, \infty)$ bijectively onto $\partial D(1, 1) \setminus\{0\}$. The change of variables $z = m(x)$ will show that 
\begin{equation}\label{muuu}
d \mu(z) = \frac{\pi}{2} \frac{\operatorname{sech}^{2}(\pi m^{-1}(z))}{|z|^2} |dz|, \quad z \in \partial D(1, 1),
\end{equation}
where $|dz|$ is arc length measure on $\partial D(1, 1)$. Parameterizing $\partial D(1, 1)$ by 
$z = 1 + e^{i \theta}, -\pi \leq \theta \leq \pi,$
a calculation shows that 
$$\frac{d \mu(1 + e^{i \theta})}{|d z|} =  \frac{\pi}{2} \frac{ \csc ^2\left(\frac{\pi }{1+e^{i \theta}}\right)}{(2 + 2 \cos \theta)}, \quad - \pi < \theta < \pi,$$
which is plotted in Figure \ref{Fig_mu}.
\begin{figure}
\begin{center}
 \includegraphics[width=.4\textwidth]{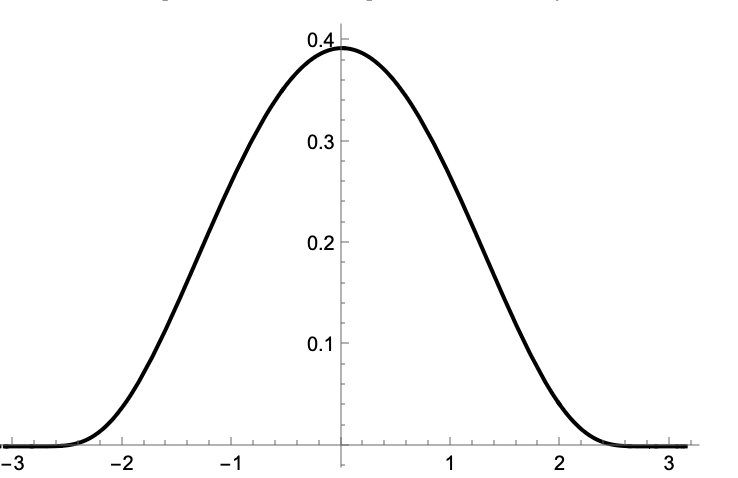}
 \caption{The plot of $\frac{d \mu(1 + e^{i \theta})}{|d z|}  = \frac{\pi  \csc ^2\left(\frac{\pi }{1+e^{i \theta}}\right)}{2 (2 \cos \theta +2)}$ for $-\pi \leq \theta \leq \pi$.}
 \label{Fig_mu}
 \end{center}
\end{figure}

We can now state our spectral representation of $C$. 

\begin{Theorem}
The Ces\`{a}ro operator $C$ on $L^2(0, 1)$ is unitarily equivalent to $M_{z} f = z f$ on $L^2(\mu)$, where $\mu$ is the measure from \eqref{muuu}.
\end{Theorem}

\begin{proof}
Let $W := \mathscr{F} \circ \Phi$ and define 
$\Psi: L^2(0, 1) \to L^2(\mu)$ by 
$$(\Psi f)(z) := \frac{(W f)(m^{-1}(z))}{(W \chi_{[0, 1]})(m^{-1}(z))} = \frac{(W f)(m^{-1}(z))}{\sqrt{\frac{\pi}{2}} \operatorname{sech}(\pi m^{-1}(z))}\quad z \in \partial D(1, 1).$$
Then, by \eqref{uuWwWW},  $\Psi$ is an isometric isomorphism, $\Psi \chi_{[0, 1]}  = \chi_{\partial D(1, 1)}$, and, for $f \in L^2(0, 1)$ and  $\mu$-almost every $z \in \partial D(1, 1)$, \eqref{sd8f8sdfhhhMMMMMmFF} says that 
\begin{align*}
(\Psi C f)(z) & = \frac{(\mathscr{F} \circ \Phi \circ C)(f(m^{-1}(z)))}{(\mathscr{F} \circ \Phi)(\chi_{[0, 1]}(m^{-1}(z)))}\\
& = M_{m(m^{-1}(z))} \frac{(\mathscr{F} \circ \Phi)(f(m^{-1}(z)))}{(\mathscr{F} \circ \Phi)(\chi_{[0, 1]}(m^{-1}(z)))}\\
& = z  \frac{(\mathscr{F} \circ \Phi)(f(m^{-1}(z)))}{(\mathscr{F} \circ \Phi)(\chi_{[0, 1]}(m^{-1}(z)))}\\
& = z (\Psi f)(z)\\
& = (M_z \Psi f)(z),
\end{align*}
which proves that $C$ is unitarily equivalent to $M_{z}$ on $L^2(\mu)$ via $\Psi$.
\end{proof}

The discussion in Remark \ref{RRR} says that $I - C$ is unitarily equivalent to the multiplication operator 
$$M_{\frac{-i x - 1/2}{-i x + 1/2}}$$
 on $L^2(-\infty, \infty)$. 
One can now make the argument, inspired by results from \cite{GPRH}, that the above multiplication operator is unitarily equivalent to the bilateral shift $h(e^{i \theta}) \mapsto e^{i \theta} h(e^{i \theta})$ on $L^2(\frac{d \theta}{2 \pi})$ (which is unitary). In particular, note that 
\begin{equation}\label{ImC}
\|I - C\| = 1.
\end{equation}

\section{Resolvent formula}

So far we  know that $\sigma(C) = \partial D(1, 1)$. For $\lambda \not = \partial D(1, 1)$, what is a formula for the resolvent  $(\lambda I - C)^{-1}$? Below we will use the fact that 
$$\lambda \in D(1, 1) \iff \Re \frac{1}{\lambda} > \frac{1}{2}$$ while 
$$\lambda \not \in \overline{D(1, 1)} \iff \Re \frac{1}{\lambda} < \frac{1}{2}.$$

\begin{Theorem}\label{resolventCC}
Let $\lambda \not \in \partial D(1, 1)$. 
\begin{enumerate}
\item If $\lambda \in D(1, 1)$, then 
$$(\lambda I - C)^{-1} g(x) = \frac{1}{\lambda} g(x) - \frac{1}{\lambda^2} \frac{1}{x^{1 - \frac{1}{\lambda}} (1 - x)^{\frac{1}{\lambda}}} \int_{x}^{1} g(t) t^{-\frac{1}{\lambda}} (1 - t)^{-1 + \frac{1}{\lambda}}dt.$$
\item If $\lambda \not \in \overline{D(1, 1)}$, then 
$$(\lambda I - C)^{-1} g(x) = \frac{1}{\lambda} g(x) + \frac{1}{\lambda^2} \frac{1}{x^{1 - \frac{1}{\lambda}} (1 - x)^{\frac{1}{\lambda}}} \int_{0}^{x} g(t) t^{-\frac{1}{\lambda}} (1 - t)^{-1 + \frac{1}{\lambda}}dt.$$
\end{enumerate}
\end{Theorem}

\begin{proof}
So far we have shown that $(\lambda I - C)^{-1}$ is a bounded operator when $\lambda \not \in \partial D(1, 1)$. To find a formula for this resolvent, we  need to solve the equation 
$$(\lambda I -  C) f = g$$
for the function $f$ in terms of $g$. This involves solving the integral equation 
$$\lambda f(x) - \frac{1}{x} \int_{0}^{x} \frac{f(t)}{1 - t} dt = g(x).$$ Multiplying through the above by $x$, then differentiating, and then manipulating the terms of the result yields 
$$f'(x) + f(x) \Big(\frac{1}{x} (1 - \frac{1}{\lambda}) - \frac{1}{\lambda} \frac{1}{1 - x}\Big)  = \frac{(x g(x))'}{\lambda x}.$$
The integrating factor $\mu(x)$ of the above (linear first order) differential equation is 
$$\mu(x) = x^{1 - \frac{1}{\lambda}} (1 - x)^{\frac{1}{\lambda}}.$$ Multiplying through by the integrating factor yields 
$$(\mu(x) f(x))' = \frac{1}{\lambda} x^{-\frac{1}{\lambda}} (1 - x)^{ \frac{1}{\lambda}} (x g(x))'.$$
Now fix $0 < a < 1$. Integrating the above expression over the interval $(a, x)$ and using integration by parts on the right hand side gives us 
\begin{equation}\label{e9r8ughhHH}
\mu(t) f(t)\Big|_{t = a}^{t = x} = \frac{1}{\lambda} t^{1 - \frac{1}{\lambda}} (1 - t)^{\frac{1}{\lambda}} g(t)\Big|_{t = a}^{t = x}  + \frac{1}{\lambda^2} \int_{a}^{x} g(t) t^{-\frac{1}{\lambda}} (1 - t)^{-1 + \frac{1}{\lambda}} dt.
\end{equation}
When $\lambda \in D(1, 1)$, equivalently $\Re \frac{1}{\lambda} > \frac{1}{2}$, we see that  
$$\lim_{a \to 1^{-}} \mu(a) f(a) = \lim_{a \to 1^{-}} \frac{1}{\lambda} a^{1 - \frac{1}{\lambda}} (1 - a)^{\frac{1}{\lambda}} g(a) = 0,$$ due to the factor $(1 - a)^{\frac{1}{\lambda}}$. Moreover, the assumption that $\Re \frac{1}{\lambda} > \frac{1}{2}$ will guarantee that $(1 - t)^{-1 + \frac{1}{\lambda}} \in L^2[0, 1]$ and so for each $0 < x < 1$ the integral 
$$\int_{1}^{x} t^{- \frac{1}{\lambda}} (1 - t)^{-1 + \frac{1}{\lambda}} g(t) dt$$
converges and 
$$\lim_{a \to 1^{-}} \int_{a}^{x}  t^{- \frac{1}{\lambda}} (1 - t)^{-1 + \frac{1}{\lambda}} g(t) dt  = \int_{1}^{x}  t^{- \frac{1}{\lambda}} (1 - t)^{-1 +\frac{1}{\lambda}} g(t) dt.$$
Putting this all together in \eqref{e9r8ughhHH} and dividing the result by the integrating factor $\mu$ yields the formula in (a).

In a similar way, when $\lambda \not \in \overline{D(1, 1)}$, equivalently $\Re \frac{1}{\lambda} < \frac{1}{2}$, then 
$$\lim_{a \to 0^{+}} \mu(a) f(a) = \lim_{a \to 0^{+}} \frac{1}{\lambda} a^{1 - \frac{1}{\lambda}} (1 - a)^{\frac{1}{\lambda}} g(a) = 0,$$ due to the factor $a^{1 - \frac{1}{\lambda}}$.  Moreover, the assumption that $\Re \frac{1}{\lambda} < \frac{1}{2}$ will guarantee that $t^{- \frac{1}{\lambda}} \in L^2[0, 1]$ and so for each $0 < x < 1$ the integral 
$$\int_{0}^{x} t^{- \frac{1}{\lambda}} (1 - t)^{-1 + \frac{1}{\lambda}} g(t) dt$$
converges and 
$$\lim_{a \to 0^{+}} \int_{a}^{x}  t^{- \frac{1}{\lambda}} (1 - t)^{-1 + \frac{1}{\lambda}} g(t) dt  = \int_{0}^{x}  t^{ - \frac{1}{\lambda}} (1 - t)^{-1 + \frac{1}{\lambda}} g(t).$$
Putting this all together in \eqref{e9r8ughhHH} and dividing the result by  the integrating factor  $\mu$ yields the formula in (b).
\end{proof}

When $\lambda = 1$ in part (a), we see that 
$$(I - C)^{-1} g(x) = g(x) + \frac{1}{1 - x} \int_{1}^{x} \frac{g(t)}{t} dt = (I - C^{*}) g(x),$$
which confirms the identity in Remark \ref{RRR}. Finally, due to the fact that $C$ is a normal operator, and thus its spectral radius is equal to its operator norm \cite[p.~244]{ConwayFA},  we have 
$$\|(\lambda I - C)^{-1}\| = \frac{1}{\operatorname{dist}(\lambda, \partial D(1, 1))}.$$
A bit of geometry will show this last quantity is equal to 
$$\frac{1}{| |1 - \lambda| - 1|}.$$

\section{Semigroups}\label{Sec4}

Several good sources for the necessary properties of semigroups of  operators we will use below are  \cite{MR1721989}\cite[Ch.~2, \S 10]{MR629828} \cite{MR710486}. Beginning with the paper \cite{MR897683}, semigroups of weighted composition operators have been a mainstay in the study of the Ces\`{a}ro operator on Banach spaces of analytic functions.
In our setting, we will explore a semigroup of weighted composition operators on $L^2(0, 1)$. 
To get us started,  fix the parameter $t \geq 0$ and define the function $\phi_t$ on $[0, 1]$ by 
$$\phi_{t}(x) : = \frac{e^{-t} x}{(e^{-t} - 1) x + 1}, \quad 0 \leq x \leq 1.$$
One can show that each function $\phi_t$ satisfies $\phi_{t}(0) = 0$, $\phi_t(1) = 1$, $\phi_t$ is  strictly increasing and continuous and thus maps $[0, 1]$ bijectively onto $[0, 1]$ (see Figure \ref{phi_POS}). 
\begin{figure}
\begin{center}
 \includegraphics[width=.4\textwidth]{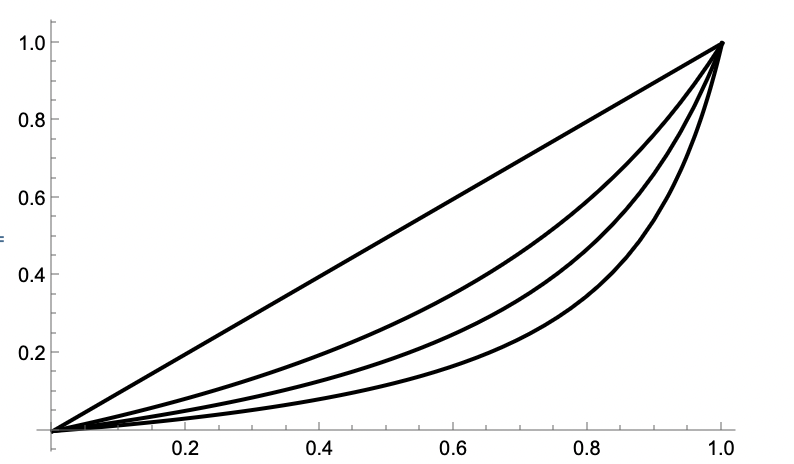}
 \caption{The graphs of the functions $\phi_0, \phi_1, \phi_{3/2}, \phi_2$.  Observe that $\phi_{t}(x) \leq	 \phi_{s}(x) \leq x = \phi_{0}(x)$ when $0 \leq s \leq t$. }
 \label{phi_POS}
 \end{center}
\end{figure}
Also observe that $\phi_{t}(x) \leq x$ for all $t \geq 0$ and $0 \leq x \leq 1$.
Now define the {\em weighted composition operator} $S_{t}$ for an  $f \in L^2(0, 1)$ by 
\begin{equation}\label{sdfsdfvVVV}
(S_{t} f)(x) := \frac{\phi_{t}(x)}{x} f(\phi_t(x)).
\end{equation}
Our first order of business is to show that $S_t$ actually defines a bounded linear  operator on $L^2(0, 1)$ and  to compute its operator norm $\|S_{t}\|$. 

\begin{Proposition}\label{0w349835345}
For each $t \geq 0$, $S_t$ is a bounded linear operator on $L^2(0, 1)$ with $\|S_{t}\| = e^{-\frac{t}{2}}$.
\end{Proposition}

\begin{proof}
The change of variables 
$$u = \phi_{t}(x), \quad du = \frac{e^{-t}}{((e^{-t} - 1) x + 1)^2} dx,$$ yields
\begin{align*}
\int_{0}^{1} |(S_{t} f)(x)|^2 dx & = \int_{0}^{1} \frac{e^{-2t}}{((e^{-t} - 1) x + 1)^2} |f(\phi_t(x))|^2 dx\\
& = e^{-t} \int_{0}^{1} |f(u)|^2 du,
\end{align*} which shows that 
$\|S_t f\| = e^{-\frac{t}{2}} \|f\|$ for all $f \in L^2(0, 1)$. The definition of the operator norm give us  $\|S_{t}\| = e^{-\frac{t}{2}}$. 
\end{proof}

A short computation shows that  $\phi_{s} \circ \phi_{t} = \phi_{s + t}$ for all $s, t \geq 0$ and thus  $\{S_t\}_{t \geq 0}$ forms a semigroup of weighted composition operators on $L^2(0, 1)$ in that $S_{0} = I$ and $S_{s} \circ S_{t} = S_{s + t}$. The strong continuity of this semigroup  is a consequence of the following. 

\begin{Proposition}\label{fdkmgnnNNNNoeewW}
$S_{t} \to I$ in the strong operator topology as $t \to 0^{+}$ in that 
$$\lim_{t \to 0^{+}} \|S_{t} f - f\| = 0$$
for each $f \in L^2(0, 1)$.
\end{Proposition}

\begin{proof}
Fix $f \in L^2(0, 1)$ and $\varepsilon > 0$. By the density of the continuous functions on $[0, 1]$ in $L^2(0, 1)$, we can choose a continuous function $g$ on $[0, 1]$ with $\|f - g\| <  \frac{\varepsilon}{2}$. By the triangle inequality and Proposition \ref{0w349835345} we have 
\begin{align*}
\|S_{t} f - f\| & \leq \|S_{t} f - S_{t} g\| + \|S_{t} g - g\| + \|g - f\|\\
& \leq \|S_{t}\| \|f - g\| + \|S_{t} g - g\| + \|g - f\|\\
& \leq e^{-\frac{t}{2}} \tfrac{\varepsilon}{2} + \|S_{t} g - g\| + \tfrac{\varepsilon}{2}\\
& \leq \epsilon +  \|S_{t} g - g\|.
\end{align*}
Since $\phi_{t}(x) \to x$ as $t \to 0^{+}$ for each $0 \leq x \leq 1$ and thus $S_{t} g \to g$ pointwise on $[0, 1]$, we can use the continuity of $g$ on $[0, 1]$, along with the dominated convergence theorem, to see that 
 $S_{t} g \to g$ in the norm of  $L^2(0, 1)$ as $t \to 0^{+}$. Using this detail and the estimate above gives us 
$$\varlimsup_{t \to 0^{+}} \|S_{t} f - f\| \leq \varepsilon.$$ The result now follows. 
\end{proof}

Our analysis below requires a formula for the adjoint of $S_{t}$. 

\begin{Proposition}\label{addjdjdjint}
For  each fixed $t \geq  0$, 
$$(S_{t}^{*} g)(x) = \frac{e^{-t}}{(1 - e^{-t}) x+ e^{-t}} g\Big(\frac{x}{(1 - e^{-t}) x + e^{-t}}\Big), \quad g \in L^2(0, 1).$$
\end{Proposition}

\begin{proof}
The change of variables $u = \phi_{t}(x)$ implies that 
$$x = \frac{u}{(1 - e^{-t}) u +e^{-t}} \;  \mbox{and} \;  
d x = \frac{e^{-t} du}{((1 - e^{-t})u + e^{-t})^2}.$$
For $f, g \in L^2(0, 1)$ we have, using the above change of variables,
\begin{align*}
\langle S_{t} f, g\rangle & = \int_{0}^{1} \frac{\phi_t (x)}{x} f(\phi_t (x)) \overline{g(x)} dx\\
& = \int_{0}^{1} f(u) \overline{\frac{e^{-t}}{((1 - e^{-t}) u + e^{-t})} g\Big(\frac{u}{(1 - e^{-t}) u + e^{-t}}\Big)} du,
\end{align*}
which yields the desired formula for $S_{t}^{*} g$. 
\end{proof}

The above discussion shows that $\{S_{t}^{*}\}_{t \geq 0}$ is also a strongly continuous semigroup of operators on $L^2(0, 1)$. Indeed, the semigroup property for $\{S_{t}^{*}\}_{t \geq 0}$ follows from that of $\{S_{t}\}_{t \geq 0}$. The strong continuity follows from a similar argument as in Proposition \ref{fdkmgnnNNNNoeewW}. This fact is not particular to the semigroup $\{S_t\}_{t \geq 0}$ and is true in greater generality \cite[p.~145]{MR629828}.
We pause here for a moment to make the following interesting observation. 

When $t < 0$, observe that the function $\phi_t$ still maps $[0, 1]$ bijectively onto $[0, 1]$ but $\phi_t(x) \geq x$ for all $0 \leq x \leq 1$ (the reverse holds when $t > 0$) (see Figure \ref{phi_NEG}).
\begin{figure}
\begin{center}
 \includegraphics[width=.4\textwidth]{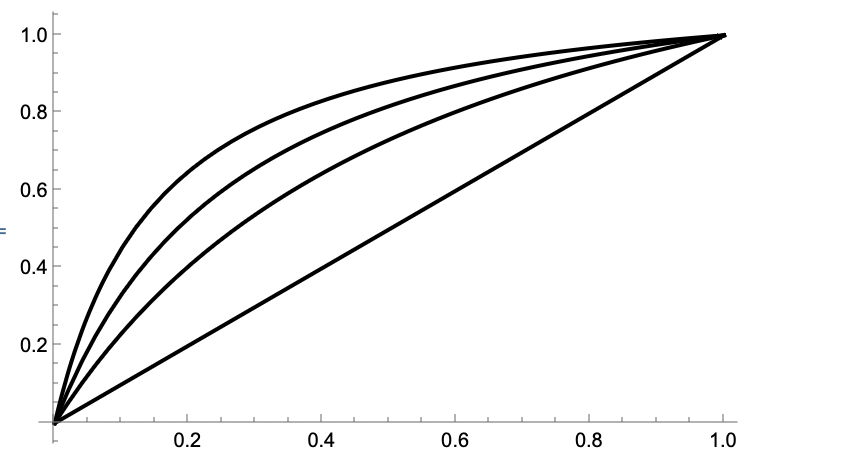}
 \caption{The graphs of $\phi_0, \phi_{-1}, \phi_{-3/2}, \phi_{-2}$.  Observe that $\phi_{s}(x) \geq \phi_{t}(x) \geq x =  \phi_{0}(x)$ when $s \leq t \leq 0$. }
 \label{phi_NEG}
 \end{center}
\end{figure}

 We will show below that for $t < 0$, the operator $S_{t}$, defined the same way as when $t > 0$,  is still a bounded operator on $L^2(0, 1)$. In fact, we can say more. 

\begin{Proposition}\label{Ut}
For each $t \geq 0$ we have the following. 
\begin{enumerate}
\item $S_{t}^{*} = e^{-t} S_{-t}$;
\item $U_{t} := e^{\frac{t}{2}} S_t$ is a unitary operator on $L^{2}(0, 1)$ with $U_{t}^{*} = U_{-t}$. 
\end{enumerate} 
\end{Proposition}

\begin{proof}
For any $g \in L^2(0, 1)$, Proposition \ref{addjdjdjint} says that 
\begin{align*}
(S_{t}^{*} g)(x) & =  \frac{e^{-t}}{(1 - e^{-t}) x+ e^{-t}} g\Big(\frac{x}{(1 - e^{-t}) x + e^{-t}}\Big)\\
& = \frac{1}{(e^{t} - 1)x + 1} g\Big(\frac{e^{t} x}{(e^{t} - 1) x + 1}\Big)\\
& = e^{-t}  \frac{e^{t}}{(e^{t} - 1)x + 1} g\Big(\frac{e^{t} x}{(e^{t} - 1) x + 1}\Big)\\
& = e^{-t} \frac{\phi_{-t}(x)}{x} g(\phi_{-t}(x))\\
& = e^{-t} (S_{-t} g)(x).
\end{align*}
This proves $(a)$.

Now observe that 
\begin{align*}
U_{t}^{*} & = (e^{\frac{t}{2}} S_{t})^{*}\\
& = e^{\frac{t}{2}} S_{t}^{*}\\
& = e^{\frac{t}{2}} e^{-t} S_{-t}\\
& = U_{-t}\\
& = U_{t}^{-1}
\end{align*}
and $U_{t}^{*} = U_{t}^{-1} = U_{-t}$, which proves $(b)$.
\end{proof}

For the semigroup $\{S_t\}_{t \geq 0}$, its {\em infinitesimal generator} $A$ is defined by 
$$A f := \lim_{t \to 0^{+}} \frac{S_{t} f - f}{t}$$ whenever this limit exists  (in the norm of $L^2(0, 1)$). Define $\mathcal{D}(A)$, the domain of $A$, to be the set of all  functions $f \in L^2(0, 1)$ for which this limit exists. Known theory about strongly continuous semigroups says that $\mathcal{D}(A)$ is a dense linear manifold of $L^2(0, 1)$ \cite[Ch.~ 2, \S 10]{MR629828}. Moreover, some calculus shows that for appropriately smooth functions $f$ we have 
\begin{equation}\label{AAAaa}
(A f)(x) =\Big( \frac{\partial}{\partial t}\Big( \frac{\phi_{t}(x)}{x} \cdot f(\phi_{t}(x))\Big)\Big)\Big|_{t = 0} = - (1 - x)(x f(x))'
\end{equation} and hence 
\begin{equation}\label{666FFAris}
\mathcal{D}(A) = \{f \in L^2(0, 1): (1 - x)(x f(x))' \in L^2(0, 1)\}.
\end{equation}
Recall from Proposition \ref{0w349835345}  that $\|S_{t}\| = e^{-\frac{t}{2}}$. There is the well-known fact  \cite[Ch.~ 2, \S 10]{MR629828}
that if $\{T_{t}\}_{t \geq 0}$ is a strongly continuous semigroup of operators on a Hilbert space with 
$$\|T_{t}\| \leq e^{w t} \; \mbox{for all $t \geq 0$,}$$  then the spectrum of its infinitesimal generator is contained in the half plane 
$$\{\lambda \in \C: \Re \lambda \leq w\}.$$ Applying this  to $T_t = S_{t}$, 
yields 
$$\sigma(A) \subset \{\lambda \in \C: \Re \lambda \leq -\tfrac{1}{2}\}.$$ 
This means that $0 \not \in \sigma(A)$ and thus if 
$$R(\lambda, A) = (\lambda I - A)^{-1}$$ is the resolvent operator, then 
setting $\lambda = 0$ and solving the resulting differential equation 
$$-(1 - x) (x f)'(x) = g(x),$$ shows that $C = R(0, A)$.
From semigroup theory \cite[p.~142]{MR629828}, the resolvent of the infinitesimal generator $A$ of $\{S_{t}\}_{t \geq 0}$ can be given, for $\Re \lambda > -\frac{1}{2}$,  by the Laplace transform formula 
 \begin{equation}\label{LaplaceT}
(\lambda I - A)^{-1}f  = \int_{0}^{\infty } e^{- \lambda t} S_{t} f dt, \quad f \in L^2(0, 1).
\end{equation}
 Thus, 
$$C f = \int_{0}^{\infty} S_{t} f dt, \quad f \in L^2(0, 1).$$ 
The spectral mapping theorem applied to the identity 
$$C = R(0, A) = (-A)^{-1}$$ \cite[p.~243]{MR1721989} says that 
$$\sigma(C) \setminus \{0\} = \sigma(R(0, A)) \setminus \{0\}.$$
Now use the fact that  $\sigma(C) = \partial D(1, 1)$ (Theorem \ref{boundedCCC}) to conclude the following. 

\begin{Corollary}
$\sigma(A) = \{\lambda \in \C: \Re \lambda = -\frac{1}{2}\}$.
\end{Corollary}

Analogous to the proof Theorem \ref{resolventCC}, we can develop a resolvent formula for $A$.

\begin{Corollary}
For $\Re \lambda >   -\frac{1}{2}$, 
$$((\lambda I - A)^{-1} f)(x) = \frac{(1 - x)^{\lambda}}{x^{\lambda + 1}} \int_{0}^{x} f(t) \frac{t^{\lambda}}{(1 - t)^{\lambda + 1}}dt.$$
For $\Re \lambda < -\frac{1}{2}$, 
$$((\lambda I - A)^{-1} f)(x) = -  \frac{(1 - x)^{\lambda}}{x^{\lambda + 1}} \int_{x}^{1} f(t) \frac{t^{\lambda}}{(1 - t)^{\lambda + 1}}dt.$$
\end{Corollary}

\begin{proof}
We know that $\sigma(A) = \{\lambda: \Re \lambda = -\frac{1}{2}\}$. Thus we can solve 
$$(\lambda I - A) f = g$$ for any (smooth enough) $f$ and $g$. So from \eqref{666FFAris} we need to solve the differential equation 
$$\lambda f(x) + (1 - x) (x f)'(x) = g(x).$$
Some manipulation of the above says we need to solve 
$$\frac{\lambda + 1 - x}{x(1 - x)} f(x) + f'(x) = \frac{g(x)}{x (1 - x)}.$$
The integrating factor $\mu(x)$ for the above first order linear differential equation is 
$$\mu(x) = \frac{x^{1 + \lambda}}{(1 - x)^{\lambda}}.$$
Multiplying by $\mu$ on both sides yields 
$$\Big( \frac{x^{1 + \lambda}}{(1 - x)^{\lambda}} f(x)\Big)' = \frac{g(x) x^{\lambda}}{(1 - x)^{\lambda + 1}}.$$
From here the proof parallels that of Theorem \ref{resolventCC}. Indeed, fix $0 < a < 1$ and integrate both sides of the above from $a$ to $x$ to get 
$$ \frac{x^{1 + \lambda}}{(1 - x)^{\lambda}} f(x) -  \frac{a^{1 + \lambda}}{(1 - a)^{\lambda}} f(a) = \int_{a}^{x} \frac{g(t) t^{\lambda}}{(1 - t)^{\lambda + 1}}dt.$$
When $\Re  \lambda > - \frac{1}{2}$ the term 
$$\frac{a^{1 + \lambda}}{(1 - a)^{\lambda}} f(a)$$ goes to zero as $a \to 0^{+}$ due to the term $a^{1 + \lambda} \to 0$. Moreover, via the Cauchy--Schwarz inequality, the integral 
$$ \int_{0}^{x} \frac{g(t) t^{\lambda}}{(1 - t)^{\lambda + 1}}dt$$ converges and
$$ \int_{a}^{x} \frac{g(t) t^{\lambda}}{(1 - t)^{\lambda + 1}}dt \to  \int_{0}^{x} \frac{g(t) t^{\lambda}}{(1 - t)^{\lambda + 1}}dt$$ as $a \to 0^{+}$. In a similar way, when $\Re \lambda < -\frac{1}{2}$, the term 
$$\frac{a^{1 + \lambda}}{(1 - a)^{\lambda}} f(a)$$ goes to zero as $a \to 1^{-}$ due to the term $(1 - a)^{-\lambda} \to 0$ as $a \to 1^{-}$. 
Moreover, the integral 
$$ \int_{x}^{1} \frac{g(t) t^{\lambda}}{(1 - t)^{\lambda + 1}}dt$$ converges and 
$$ \int_{x}^{a} \frac{g(t) t^{\lambda}}{(1 - t)^{\lambda + 1}}dt \to  \int_{x}^{1} \frac{g(t) t^{\lambda}}{(1 - t)^{\lambda + 1}}dt$$ as $a \to 1^{-}$. 
These two cases give us the desired formulas for $(\lambda I - A)^{-1}$. 
\end{proof}

In the previous corollary notice that when $\lambda = 0$ we get 
$(-A)^{-1} = C$ while when 
 $\lambda = -1$ we get 
$(I + A)^{-1} = C^{*}$.

\section{Invariant subspaces}

This next result is inspired by  analogous results about the invariant subspaces for the classical Ces\`{a}ro operator $\mathscr{C}$ on $H^2$ \cite{MR4757014} and the Hardy averaging operator $H$ on $L^2(0, 1)$  from \eqref{Haaa}   \cite{GPRH}. Recall the semigroup of weighted composition operators $\{S_t\}_{t \geq 0}$ on $L^2(0, 1)$ from \eqref{sdfsdfvVVV}.

\begin{Theorem}\label{ISCe}
For a closed subspace $\mathcal{M}$ of $L^2(0, 1)$, the following are equivalent.
\begin{enumerate}
\item $C \mathcal{M} \subset \mathcal{M}$; 
\item $S_{t} \mathcal{M} \subset \mathcal{M}$ for all $t \geq 0$. 
\end{enumerate}
\end{Theorem}

Before getting to the proof, we need a standard  technical detail concerning dilations of semigroups. 

\begin{Lemma}\label{lcc}
Suppose $\{T_{t}\}_{t \geq 0}$ is a strongly continuous semigroup of bounded Hilbert space operators with infinitesimal generator $A$  and $\lambda, \mu \geq 0$. Then 
$$\widetilde{T}_{t} := e^{\lambda t} T_{\mu t}, \quad t \geq 0, $$
is also a strongly continuous semigroup of operators with infinitesimal  generator $\widetilde{A}$ satisfying 
$\widetilde{A} = \mu A + \lambda I.$
\end{Lemma}

\begin{proof}
It suffices to verify the formula for the infinitesimal  generator  $\widetilde{A}$. Indeed, 
\begin{align*}
\widetilde{A} f&= \lim_{t \to 0^{+}} \frac{\widetilde{T}_{t} f - f}{t}\\
& = \lim_{t \to 0^{+}} \frac{e^{\lambda t} T _{\mu t} f - f}{t}\\
& = \lim_{t \to 0^{+}} \frac{e^{\lambda t} T_{\mu t} f - T_{\mu t} f + T_{\mu t} f - f}{t}\\
& = \lim_{t \to 0^{+}} \frac{e^{\lambda t} - 1}{t} T_{\mu t} f + \lim_{t \to 0^{+}} \frac{T_{\mu t} f - f}{t}\\
& = \lambda f + \mu \lim_{t \to 0^{+}} \frac{T_{\mu t} f - f}{\mu t}\\
& = \lambda f + \mu A f,
\end{align*}
which proves the desired formula. 
\end{proof}

\begin{proof}[Proof of Theorem \ref{ISCe}]
Proposition \ref{0w349835345} says that $\{S_t\}_{t \geq 0}$ is a strongly continuous semigroup with $\|S_t\| = e^{-\frac{t}{2}}$. Now define 
$$\widetilde{S_t} =  e^{t} S_{2 t}, \quad t \geq 0,$$
and observe that $\|\widetilde{S_{t}}\| = 1$ and so $\{\widetilde{S_t}\}_{t \geq 0}$ is a {\em contractive} semigroup. 

By the general theory of contractive semigroups \cite[p.~146]{MR629828}, a closed subspace $\mathcal{M}$ of $L^2(0, 1)$ satisfies 
$\widetilde{S_t} \mathcal{M} \subset \mathcal{M}$ if and only if $V \mathcal{M} \subset \mathcal{M}$, where 
$$V := (\widetilde{A} + I)(\widetilde{A} - I)^{-1}$$ is the {\em cogenerator} of $\{\widetilde{S_t}\}_{t \geq 0}$ and $\widetilde{A}$ is the infinitesimal generator of $\{\widetilde{S_t}\}_{t \geq 0}$. 

We also have the identity 
$\widetilde{A} = 2 A + I,$ where $A$ is the infinitesimal generator of $\{S_t\}_{t \geq 0}$ (Lemma \ref{lcc}). Combining this with the identity $(-A)^{-1} = C$, as observed earlier in \eqref{LaplaceT}, we see that 
\begin{align*}
V &= (\widetilde{A} + I)(\widetilde{A} - I)^{-1}\\
& = (2 A + 2 I) (2A)^{-1}\\
& = I + A^{-1}\\
& = I - C.
\end{align*}
Since the operators $V = I - C$ and $C$ have the same invariant subspaces, it follows that  $S_{t} \mathcal{M} \subset \mathcal{M}$ for all $t \geq 0$ if and only if $C \mathcal{M} \subset \mathcal{M}$. 
\end{proof}

Certain obvious invariant subspaces for $C$ are 
$$\mathcal{M}_{a} := \{f \in L^2(0, 1): \mbox{$f = 0$ a.e.~on  $[0, a]$}\}, \quad 0 \leq a \leq 1.$$
Notice how $S_{t} \mathcal{M}_{a} \subset \mathcal{M}_{a}$ due to the fact that $\phi_{t}(x) \leq x$ for all $t \geq 0$ (Figure \ref{phi_POS}). It is known  \cite{MR31110, MR92124} that $\mathcal{M}_{a}$, $0 \leq a \leq 1$,   are the {\em only} invariant subspaces for the classical Volterra operator from \eqref{Volt}.  We will explore other types of invariant subspaces for $C$ below.

Theorem \ref{ISCe} says that a closed subspace $\mathcal{M} \subset L^2(0, 1)$ is invariant for $C$ if and only if it is invariant for every $S_t$, $t \geq 0$. We now use an idea from \cite{MR2858500, MR3057690} to explore the complexity of the invariant subspaces for the individual unitary operators
$$U_{t} = e^{\frac{t}{2}} S_{t}$$ from Proposition \ref{Ut} (which are the same as those for $S_t$). Though this analysis might appear a bit of a diversion from our original discussion, which considers the common invariant subspaces for the semigroup $\{S_{t}\}_{t \geq 0}$, we feel that $U_t$ is an interesting operator in its own right. 

Fix  $t > 0$ and recall  the mapping on $[0, 1]$ defined by 
$$\phi_{-t}(x) = \frac{e^t x}{(e^t - 1) x + 1}.$$
As observed earlier, $\phi_{-t}(x) \geq x$ (see Figure \ref{phi_NEG}). Set 
$$\psi_{-t}(x) := \frac{\phi_{-t}(x)}{x}$$ and observe how this notation says that 
$$S_{-t} f = \psi_{-t} \cdot (f \circ \phi_{-t}).$$
For each $n \in \Z$, define 
$$a_{n} := \phi_{-t n}(\tfrac{1}{2}) = \frac{1}{1 + e^{-nt}}$$ and notice that 
$$0 < \ldots a_{-2} < a_{-1} < a_0 < a_1 < a_2 < \ldots < 1$$
and the sequence $(a_n)_{n \in \Z}$ satisfies 
$$\lim_{n \to \infty} a_{n} = 1 \; \mbox{while} \; \lim_{n \to -\infty} a_n = 0.$$ As a result, 
$$\bigcup_{n \in \Z} [a_{n}, a_{n + 1}] = (0, 1).$$
It is important to note that the $a_n = a_n(t)$ depend on $t$.

For $f \in L^2(0, 1)$ define $V_t f$ to be the two sided sequence of functions 
$$V_t f := \Big((f \circ \phi_{-n t}) \cdot \sqrt{\phi_{-nt}'}\Big)_{n \in \Z}$$
and observe that 
\begin{align*}
\int_{0}^{1} |f(x)|^2 dx & = \sum_{n \in \Z}\int_{0}^{1} |f \cdot \chi_{[a_n, a_{n + 1}]}|^2 dx\\
& = \sum_{n \in \Z} \int_{a_0}^{a_{1}} |f \circ \phi_{-nt}|^2 \phi_{-nt}' dx\\
&= \sum_{n \in \Z} 
\Big\|f \circ \phi_{-nt} \sqrt{\phi_{-nt}'}\Big\|^{2}_{L^2[a_{0}, a_{1}]}.
\end{align*}
This shows that $V_t$ defines an isometric isomorphism from $L^2(0, 1)$ onto the vector-valued sequence space $\ell^2(\Z, L^2[a_0, a_{ 1}])$, the bilateral sequence of functions, each in $L^2[a_0, a_1]$, with square summable $L^2$ norms. 
Again, notice that $a_n$ depends on $t$. In particular, 
$$L^{2}[a_0, a_1] = L^{2}[\tfrac{1}{2}, \tfrac{1}{1 + e^{-t}}].$$

For $f \in L^2(0, 1)$ we also have
\begin{align*}
V_t \circ S_{-t} f & = V_t ((f \circ \phi_{-t}) \cdot \psi_{-t})\\
& = \Big((f \circ \phi_{-t (n + 1)}) \cdot (\psi_{-t} \circ \phi_{-nt}) \cdot \sqrt{\phi_{-nt}'}\Big)_{n \in \Z}.
\end{align*}
Moreover, a calculation shows that 
$$\psi_{-t} \circ \phi_{-nt} = \frac{\phi_{-t} \circ \phi_{-nt}}{\phi_{-nt}} =\frac{\phi_{-(n+ 1) t}}{\phi_{-nt}}$$
and 
$$\sqrt{\phi_{-nt}'} = \frac{e^{\frac{nt}{2}}}{((e^{n t} - 1)x + 1)^2} = e^{-\frac{nt}{2}} \psi_{-nt}.$$
Thus,
\begin{align*}
(\psi_{-t} \circ \phi_{-nt}) \cdot \sqrt{\phi_{-nt}'} & = \frac{\phi_{-(n+ 1) t}}{\phi_{-nt}} e^{\frac{nt}{2}} \psi_{-nt}\\
& = (\psi_{-t} \circ \phi_{-nt} ) \cdot \sqrt{\phi_{-nt}'} \\
& = \frac{\phi_{-(n+ 1) t} \cdot \frac{1}{x}}{\phi_{-nt} \cdot \frac{1}{x}} e^{\frac{nt}{2}} \psi_{-nt}\\
& = \frac{\psi_{-(n + 1) t}}{\psi_{-nt}} e^{-\frac{nt}{2}} \psi_{-nt}\\
& = e^{-\frac{nt}{2}}\psi_{-(n + 1) t}.
\end{align*}
Putting this all together yields 
\begin{equation}\label{VtSt}
V_{t} \circ S_{-t} f = \Big((f \circ \phi_{-(n + 1) t}) \cdot e^{-\frac{nt}{2}} \cdot  \psi_{-(n + 1) t}\Big)_{n \in \Z}.
\end{equation}

Now consider the unitary operator 
$$B_t: \ell^2(\Z, L^{2}[a_{0}, a_{1}]) \to  \ell^2(\Z, L^{2}[a_{0}, a_{1}])$$
defined by 
$$B_t\Big((f \circ \phi_{-nt}) \cdot \sqrt{\phi_{-nt}'}\Big)_{n \in \Z} := \Big((f \circ \phi_{-(n + 1)t}) \cdot \sqrt{\phi_{-(n + 1)t}'}\Big)_{n \in \Z},$$
and note that $B_t$ is a {\em bilateral backward shift}. Finally, by the calculations above, 
\begin{align*}
(e^{\frac{t}{2}} \circ B_{t})\circ V_{t} f & = e^{\frac{t}{2}} B_t \Big((f \circ \phi_{-nt}) \cdot \sqrt{\phi_{-nt}'}\Big)_{n \in \Z}\\
& = \Big(e^{\frac{t}{2}} (f \circ \phi_{-(n + 1) t}) \cdot \sqrt{\phi_{-(n + 1) t}'}\Big)_{n \in \Z}\\
& = \Big(e^{\frac{t}{2}} (f \circ \phi_{-(n + 1) t}) \cdot e^{-\frac{(n + 1)t}{2}} \cdot  \psi_{-(n +1) t}\Big)_{n \in \Z}\\
& = \Big((f \circ \phi_{-(n + 1) t}) \cdot e^{-\frac{nt}{2}} \cdot  \psi_{-(n + 1) t}\Big)_{n \in \Z},
\end{align*}
which is equal to the sequence $V_{t} \circ S_{-t} f$ from \eqref{VtSt}. This yields the following.

\begin{Theorem}
For each $t > 0$, $S_{-t}$ is unitarily equivalent to $e^{\frac{t}{2}} B_{t}$, where $B_t$ is the backward bilateral shift on $\ell^2(\Z, L^2[a_0, a_{ 1}])$.
\end{Theorem}

Since, for $t > 0$, $S_{t}^{*} = e^{-t} S_{-t}$ (Proposition \ref{Ut}), we conclude that $S_{t}$ is unitarily equivalent to $e^{-\frac{t}{2}} B_{t}^{*}$. Moreover, $B_{t}^{*}$ will be the {\em bilateral forward shift} on the vector-valued sequence space $\ell^2(\Z, L^2[a_0, a_{1}])$. Since $U_{t} = e^{\frac{t}{2}} S_{t}$, we obtain the following.

\begin{Corollary}
For each  $t > 0$, the unitary operator $U_{t}$ is unitarily equivalent to $B_{t}^{*}$ on $\ell^2(\Z, L^2[a_0, a_{1}])$.
\end{Corollary}

 This seems to say that the invariant subspaces of $U_t$ are complicated and are as rich as those for  the bilateral forward shift $B^{*}_{t}$. 

\begin{Example}
Consider the subspace $\mathcal{N} \subset \ell^2(\Z, L^2[a_0, a_{1}])$ consisting of the $(g_{n})_{n \in \Z}$
such that $g_n = 0$ for all $n < -1$ and notice that $\mathcal{N}$ is closed and $B_{t}^{*} \mathcal{N} \subset \mathcal{N}$. Then, via $V_{t}$ from above, 
$$0 = g_{n} = (f \circ \phi_{-nt}) \cdot \sqrt{\phi_{-nt}'} \; \mbox{on $[a_0, a_{ 1}]$}, \quad n < -1.$$
and so 
$$V_{t}^{-1} \mathcal{N} = \{f \in L^2(0, 1): f \circ \phi_{m t} = 0 \; \mbox{on $[a_{0}, a_{1}]$} \; \mbox{for all $m \geq 1$}\}.$$
A moment of thought shows $V_{t}^{-1} \mathcal{N}$ turns out to be
$$\{f \in L^2(0, 1): \mbox{$f = 0$ a.e. on $[0, a_{0}]$}\}.$$ Notice how this is the invariant $\mathcal{M}_{a_0}$ for $C$ discussed earlier. \end{Example}

The final part of the discussion of the complexity of the invariant subspaces for the Ces\`{a}ro operator $C$ on $L^{2}(0, 1)$, as well as their complete description, comes from the Beurling--Lax theorem. In order to get there, we need a few preliminaries. Recalling the functions
$$\phi_{t}(x) : = \frac{e^{-t} x}{(e^{-t} - 1) x + 1}, \quad t \geq 0,$$ 
and the mapping
$$x(u) = \frac{1}{1 + e^u}$$ from \eqref{xxxuuu},
a brief calculation shows that 
\begin{equation}\label{linear}
\phi_{t}(x(u)) = x(u + t).
\end{equation}

Now define the translation semigroup $\{T_{t}\}_{t \geq 0}$ on $L^2(-\infty, \infty)$ by 
$$(T_{t} g)(u) = g(u + t). $$ The isometric isomorphism $\Phi$ from Proposition \ref{sdfsdf88} and \eqref{linear} tell us that 
\begin{align*}
(\Phi \circ S_{t} f)(u) & = \frac{e^{\frac{u}{2}}}{1 + e^{u}} \frac{\phi_{t}(x(u))}{x(u)} f(\phi_{t}(x(u)))\\
& = \frac{e^{\frac{u}{2}}}{1 + e^{u + t}} f(x(u + t))\\
& = e^{-\frac{t}{2}} \frac{e^{\frac{u + t}{2}}}{1 + e^{u + t}}  f(x(u + t))\\
& = e^{-\frac{t}{2}} (T_{t}  \circ \Phi f)(u).
\end{align*}
We summarize the above discussion with the following result. 

\begin{Proposition}
$\Phi \circ S_{t} \circ \Phi^{-1} = e^{-\frac{t}{2}}T_{t}$ for all $t \geq 0$.
\end{Proposition}

Combining the previous proposition with Theorem \ref{ISCe} yields a connection between the invariant subspaces for $C$ and the common invariant subspaces for the translation semigroup $\{T_{t}\}_{t \geq 0}$. 

\begin{Corollary}\label{Cocococ999}
A closed subspace $\mathcal{M}$ of $L^2(0, 1)$ is invariant for the Ces\`{a}ro operator $C$ if and only if $\Phi \mathcal{M}$ is invariant every $T_{t}, t \geq 0$. 
\end{Corollary}

This leads us to a discussion of the closed subspaces $\mathcal{N}$ of $L^2(-\infty, \infty)$ which are $T_{t}$-invariant for all $t \geq 0$. These can be described by several paths. Let us use a theorem of Helson and consider the semigroup of operators 
$$\mathcal{S}_{t}: L^2(-\infty, \infty) \to L^2(-\infty, \infty), \quad \mathcal{S}_{t} f(x) = e^{i t x} f(x), \quad t \in (-\infty, \infty).$$
Subspaces $\mathcal{N}$ of $L^2(-\infty, \infty)$ which are $\mathcal{S}_t$- invariant for all $t > 0$ come in two basic types. There are the {\em doubly invariant} ones, those $\mathcal{N}$ for which are $\mathcal{S}_t$-invariant for all $t \in (-\infty, \infty)$ and the {\em simply invariant} ones, those $\mathcal{N}$ which are $\mathcal{S}_{t}$-invariant for all $t > 0$ but are not $\mathcal{S}_{t}$-invariant for some $t < 0$. A well-known theorem of Helson  \cite[Lec. V]{MR0171178} says that $\mathcal{N}$ is doubly invariant if and only if $\mathcal{N} = \chi_{A} L^2(-\infty, \infty)$ for some measurable subset $A$ of $(-\infty, \infty)$ and simply invariant if and only if $\mathcal{N} = q \mathcal{H}^2$, where $q$ is a measurable unimodular function on $(-\infty, \infty)$ and 
$\mathcal{H}^2 := \mathscr{F}^{-1} (\chi_{(0, \infty) }L^2(-\infty, \infty))$ is the Hardy space of the real line. 

The above discussion will yield the $T_{t}$-invariant subspaces for all $t > 0$ as follows. A simple computation yields 
$T_{t} = \mathscr{F}^{-1} \circ \mathcal{S}_t \circ \mathscr{F}$. Thus, $\mathcal{N}$ is $T_t$-invariant if and only if $\mathscr{F} (\mathcal{N})$ is $\mathcal{S}_t$-invariant. This brings us to the result that if $\mathcal{N}$ is $T_{t}$-invariant, then 
$$\mathcal{N}  = \mathscr{F}^{-1} (\chi_{A} L^2(-\infty, \infty))$$ for some measurable subset $A$ of $(-\infty, \infty)$, precisely when $\mathcal{N}$ is $T_{t}$-invariant for all $t \in (-\infty, \infty)$, or 
$$\mathcal{N} = \mathscr{F}^{-1} (q \mathcal{H}^2)$$ for some measurable unimodular function $q$, precisely when $\mathcal{N}$ is $T_{t}$-invariant for all $t  > 0$ but not $T_{t}$-invariant for some $t < 0$. 


Combine the above discussion to Corollary \ref{Cocococ999} to obtain the following. 

\begin{Corollary}
Let $\mathcal{M}$ be a closed subspace of $L^2(0, 1)$. Then $\mathcal{M}$ is invariant for the Ces\`{a}ro operator if and only if either 
$$\mathcal{M} = \Phi^{-1} (\mathscr{F}^{-1}(\chi_{A} L^2(-\infty, \infty)))$$
for some measurable set $A \subset (-\infty, \infty)$ or 
$$\mathcal{M} = \Phi^{-1}  (\mathscr{F}^{-1} (q \mathcal{H}^2))$$
for some measurable unimodular function $q$ on $(-\infty, \infty)$. 
\end{Corollary}

Let us work a few examples. 

\begin{Example}
For fixed $0 < a \leq 1$,
$$\mathcal{M}_{a} := \{f \in L^2(0, 1): \mbox{$f = 0$ almost everywhere on $[0, a]$}\}$$
is certainly a closed invariant subspace for the Ces\'{a}ro operator on $L^2(0, 1)$. Let us view $\mathcal{M}_{a}$ as the second type of subspace in the previous corollary. For any $b \in (-\infty, \infty)$ let $q$ be the unimodular function on $(-\infty, \infty)$ defined by  $q(x) = e^{-i bx}$. An integral substitution will show that 
$\mathscr{F}^{-1} (q \mathcal{H}^2),$
 is the set of $L^2(-\infty, \infty)$ functions which vanish almost everywhere on $(b, \infty)$
and thus is invariant for all $T_{t}$, $t \geq 0$. By the formula for $\Phi^{-1}$ from \eqref{gginnd}, these will  correspond to the $L^2(0, 1)$ functions $f(x)$ which vanish when 
$$\log \frac{1 - x}{x} \geq b,$$
that is, when 
$$0 \leq x \leq \frac{1}{1 + e^{b}}.$$
Now adjust $b$ so that $(1 + e^{b})^{-1} = a$ and so $\mathcal{M}_a = \Phi^{-1}(\mathscr{F}^{-1} (e^{-ib x} \mathcal{H}^2))$.
\end{Example}

\begin{Example}
Suppose that $A$ is a measurable subset of $(-\infty, \infty)$. From the previous corollary,  $$\mathcal{M} := \Phi^{-1}(\mathscr{F}^{-1}(\chi_{A} L^2(-\infty, \infty)))$$
is a closed invariant subspace for the Ces\`{a}ro operator on $L^2(0, 1)$. A calculation using the formula \eqref{gginnd} for $\Phi^{-1}$ says that 
$\mathcal{M}$ is the subspace of functions in $L^2(0, 1)$ of the form
$$x \mapsto \frac{1}{x} \sqrt{\frac{x}{1 - x}} (\mathscr{F}^{-1} f)\Big(\log \frac{1 - x}{x}\Big), \quad f \in L^2(A).$$
\end{Example}

\bibliographystyle{plain}

\bibliography{references}

\end{document}